\documentclass[letterpaper,11pt]{amsart}

\usepackage[all]{xy}                        %

\CompileMatrices                            % Faster

\UseTips                                    % Use

\input xypic
\usepackage[bookmarks=true]{hyperref}       % Hyperref
\usepackage{amssymb,latexsym,amsmath,amscd}
\usepackage{xspace}

\usepackage{graphicx}
\reversemarginpar

\theoremstyle{plain}
\newtheorem{theorem}{Theorem}[section]
\newtheorem*{theorem*}{Theorem}
\newtheorem{proposition}[theorem]{Proposition}
\newtheorem{corollary}[theorem]{Corollary}
\newtheorem{lemma}[theorem]{Lemma}

\theoremstyle{definition}

\newtheorem{remark}[theorem]{Remark}

\newcommand{\enm}[1]{\ensuremath{#1}}          %
\newcommand{\op}[1]{\operatorname{#1}}
\newcommand{\cal}[1]{\mathcal{#1}}

\newcommand{\ZZ}{\enm{\mathbb{Z}}}

\newcommand{\PP}{\enm{\mathbb{P}}}

\newcommand{\Aa}{\enm{\cal{A}}}

\newcommand{\Cc}{\enm{\cal{C}}}

\newcommand{\Ee}{\enm{\cal{E}}}
\newcommand{\Ff}{\enm{\cal{F}}}
\newcommand{\Gg}{\enm{\cal{G}}}
\newcommand{\Hh}{\enm{\cal{H}}}
\newcommand{\Ii}{\enm{\cal{I}}}
\newcommand{\Jj}{\enm{\cal{J}}}

\newcommand{\Ll}{\enm{\cal{L}}}

\newcommand{\Nn}{\enm{\cal{N}}}
\newcommand{\Oo}{\enm{\cal{O}}}

\newcommand{\Tt}{\enm{\cal{T}}}

\renewcommand{\phi}{\varphi}
\renewcommand{\theta}{\vartheta}
\renewcommand{\epsilon}{\varepsilon}

\newcommand{\codim}{\op{codim}}

\renewcommand{\to}[1][]{\xrightarrow{\ #1\ }}

\newcommand{\old}[1]{}

\newcommand{\Ji}{\mathcal {J}}

\begin{document}

%\layout
\title[Reflexive sheaves]{Reflexive and spanned sheaves on $\PP^3$}
\author{E. Ballico, S. Huh and F. Malaspina}
\address{Universit\`a di Trento, 38123 Povo (TN), Italy}
\email{edoardo.ballico@unitn.it}
\address{Department of Mathematics, Sungkyunkwan University \\
Cheoncheon-dong, Jangan-gu \\
Suwon 440-746, Korea}
\email{sukmoonh@skku.edu}
\address{Politecnico di Torino, Corso Duca degli Abruzzi 24, 10129 Torino, Italy}
\email{francesco.malaspina@polito.it}
\thanks{The second author is supported by Basic Science Research Program 2012-0002904 through NRF funded by MEST}

\begin{abstract}
We investigate the reflexive sheaves on $\PP^3$ spanned in codimension 2 with very low first Chern class $c_1$. We also give the sufficient and necessary conditions on numeric data of such sheaves for indecomposabiity. As a by-product we obtain that every reflexive sheaf on $\PP^3$ spanned in codimension 2 with $c_1=2$ is spanned. 
\end{abstract}

\maketitle

\section{Introduction}
Reflexive sheaves play an important role in algebraic geometry as a natural generalization of vector bundles and the investigation on them over projective spaces was deeply initiated by Hartshorne in \cite{Hartshorne1}. One of the basic questions on vector bundles is their globally generatedness and their classification problem for lower first Chern classes was recently answered by several authors in many ways \cite{am}\cite{m}\cite{ce}\cite{e}\cite{SU}\cite{SU2}. The authors of this article investigated the globally generated vector bundles on a smooth quadric threefold and give a numerical criterion for indecomposability \cite{BHM}\cite{BHM1}. 

So the natural question is on the classification of reflexive sheaves on projective spaces with lower first Chern class with certain extra conditions, for example spannedness in codimension 2. 

In this paper we investigate
the existence of a rank $r$ reflexive sheaf on the projective variety $\PP^n$ with $n\ge 3$ and very low first Chern class $c_1=c_1(\Ff )$ that is spanned in codimension 2, i.e. such that the evaluation
map $e_{\Ff}: H^0(\Ff )\otimes \Oo _X \to \Ff$ is surjective outside finitely many subvariety of codimension at least $ 3$.

Our main ingredient in this article is the sheaf $\Ff_{k,n,L}$ defined by
$$\Ff_{k,n,L}:=\mbox{coker}(~j:\Oo_{\PP^n}(-1) \to \Oo_{\PP^n}^{\oplus (k+1)}~)$$
where the map $j$ is given by $k+1$ linearly independent sections of $\Oo_{\PP^n}(1)$ with a linear subspace $L\subset \PP^n$ as their common zero locus. It is a torsion-free sheaf on $\PP^n$ with rank $k$ and singular locus $L$ (see Section 3). When $L=\{P\}$ is a single point space, then we will simply denote it by $\Ff_{k,n,P}$. 

The main result of this article is the following theorem:
\begin{theorem}
There exists an indecomposable and non-locally free reflexive sheaf $\Ff$ of rank $r\ge2$ on $\PP^3$ with $c_1\le 2$ and spanned in codimension $2$ if and only if $(c_1,c_2, r)$ is one of the followings:
\begin{align*}
&(1,1,2)~;~\Ff \cong \Ff_{2,3,P} \text{ with } P\in \PP^3,\\
&(2,2,2) ~;~ \Ff \in \mathfrak{A}, \text{a 6-dimensional irreducible variety},\\
& (2,3,r) ~;~r\in \{2,3,5\} \text{ and } \\
&(2,4, r) ~;~r\in \{2,3,4,5,6,7,8\}.
\end{align*}
The exceptional case is $(c_1,c_2,r)=(2,3,4)$ when there are only $\Ff_{2,3,P}\oplus \Ff_{2,3,P'}$ with $P, P'\in \PP^3$. 
\end{theorem}

Indeed, we do not only give a numerical criterion for the existence of indecomposable and non-locally free sheaves, but also give a complete description of reflexive sheaves spanned in codimension 2 including the locally free sheaves. In particular, we obtain that there is no indecomposable reflexive sheaf of rank higher than $7$ on $\PP^3$ with $c_1\le 2$ and spanned in codimension $2$. 
As an automatic consequence, we obtain the following result :
\begin{corollary}\label{OOO2}
Let $\Ff$ be a reflexive sheaf on $\PP^3$ with $c_1(\Ff)=2$ and $\dim V(\Ff ) \le 0$. Then we have $V(\Ff )=\emptyset$, i.e. $\Ff$ is spanned.
\end{corollary}

Based on the classification on $\PP^3$, we also prove that every vector bundle on $\PP^n$, $n\ge 3$ with $c_1=2$  and spanned in codimension $2$, is also spanned everywhere (see Section 5). 

%%%%%%%%%%%%%%%%%%%%%%%%%%%%%%%%%%%%
\section{Preliminaries}\label{S0}
Let $\Ff$ be a reflexive sheaf of rank $r$ on a projective space $\PP^n$. For our technical reason, we define two subvaieties associated to $\Ff$ :
\begin{align*}
\mbox{Sing}(\Ff)&:=\{P\in \PP^n~ |~ \Ff \text{ is not locally free at }P\}\\
V(\Ff)&:=\{ P\in \PP^n ~|~ e_{\Ff} \text{ is not surjective at }P\}
\end{align*}
where $e_{\Ff} : H^0(\Ff)\otimes \Oo_{\PP^n} \to \Ff$ is the evaluation map of $\Ff$. 

Note that the codimension of $\mbox{Sing}(\Ff)$ is at least $3$ since $\Ff$ is reflexive. In particular, every reflexive sheaf on $\PP^2$ is locally free. We also get that $\Ff$ is a line bundle when $r=1$ due to Proposition 1.9 in \cite{Hartshorne1}. Thus let us assume that $n\geq 3$ and $r\ge 2$. Set $c_1= c_1(\Ff)\in \mbox{Pic}(\PP^n)$ and then we can identify $c_1$ with an integer, i.e. $c_1\in \ZZ$. 

Moreover we assume that $V(\Ff )$ has codimension at least $3$ in $\PP^n$, which implies that $c_1\geq 0$. Taking $r-1$ general sections of $\Ff$, we get an
inclusion $j: \Oo _{\PP^n}^{\oplus (r-1)} \to \Ff $ and let us set $\Jj := \mbox{coker}(j)$. 

\begin{lemma}\label{oo3}
$\Jj$ has no torsion and $\Jj \cong \Ii_C(c_1)$ for some closed subscheme $C\subset \PP^n$ with codimension at least $2$.
\end{lemma} 

\begin{proof}

The definition of $\Ji$ gives the exact sequence
\begin{equation}\label{eqoo1}
0 \to \Oo _{\PP^n}^{\oplus (r-1)} \stackrel{j}{\to} \Ff  \stackrel{u}{\to} \Jj \to 0.
\end{equation}
Since $\Ff$ is reflexive, we have $\mbox{depth}(\Ff_P)\geq 2$ for every $P\in \PP^n$. Since the depth of $\Oo_{\PP^n}$ at $P\in \PP^n$ is at least 2, the sequence (\ref{eqoo1}) gives that $\mbox{depth}(\Jj_P)\geq 1$ at each point of $\PP^n$.

 Let $\tau \subset \Jj$ be the torsion of $\Jj$ and set $\Aa:= u^{-1}(\tau )$. $\Jj$ has no torsion if and only if $\Aa = j(\Oo_{\PP^n}^{\oplus (r-1)})$. Fix a general complete intersection smooth curve $C \subset \PP^n$, e.g. take as $C$ a general line. For such a general curve $C$, we have 
 $C\cap V(\Ff ) =\mbox{Sing}(\Ff|_C)=\emptyset$.
Let $V_C$ be the image of $H^0(\Ff )$ by the restriction map $H^0(\Ff )\to H^0(C,\Ff \vert _C)$. By the choice of $C$, $V_C$ spans the vector bundle $\Ff \vert _C$. We fix $C$ and then take $r-1$ general sections of $\Ff$ to get $u$. For general sections, we get that $j(\Oo _C^{\oplus (r-1)})$ is a subbundle of $\Ff \vert_ C$. Hence $\tau \cap C =\emptyset$. Since $C$ is general,
we get that $\tau$ has support of codimension at least $2$, i.e.
$\Aa/j(\Oo_{\PP^n}^{\oplus (r-1)})$ has support of codimension at least $2$. Since $j(\Oo_{\PP^n}^{\oplus (r-1)})$ is reflexive and $\Aa$ has no torsion, so we get $\Aa = j(\Oo_{\PP^n}^{\oplus (r-1)})$. 

Hence $\Jj$ has no torsion, i.e.
the natural map $\Ji \to \Ji ^{\vee \vee}$ is injective. Since $\Ji ^{\vee \vee}$ has rank $1$ and it is reflexive, it is a line bundle by Proposition 1.9 in \cite{Hartshorne1}. Hence $\Jj \cong \Ii_C(c_1)$ for some closed subscheme $C\subset \PP^n$ with codimension at least $2$.
\end{proof}

Lemma \ref{oo3} gives an exact sequence for a reflexive sheaf $\Ff$ :
\begin{equation}\label{eqa1}
0 \to \Oo _{\PP^n}^{\oplus (r-1)} \to \Ff \to \Ii _C(c_1)\to 0
\end{equation}
where $C$ is a subscheme of $\PP^n$ with codimension at least 2. 

\begin{lemma}\label{oo4}
The subscheme $C$ is locally Cohen-Macaulay with no embedded component such that $C_{red}$ has pure codimension 2. Moreover, we have $\mathrm{Sing}(C)\subseteq  V(\Ff )\cup \mathrm{Sing}(\Ff)$. If $n\ge 3$, then $C$ is reduced.
\end{lemma}

\begin{proof}
From the sequence (\ref{eqa1}), the sheaf $\Ii _C$ has projective dimension at least 2. Since $\PP^n$ is a regular scheme, the Serre formula gives that $\mathcal {O}_C$ has depth at least $ n-2$. Since
$C_{red}$ contains no hypersurface, the local ring $\mathcal {O}_{C,P}$ also has depth at least $n-2$ for each $P\in C_{red}$. Since $\mathcal {O}_{C,P}$ has
dimension at most $n-2$, it must be a Cohen-Macaulay local ring of dimension $n-2$. Hence $C$ has pure dimension $n-2$ and no embedded component. A form of the Bertini theorem
says that $C$ is smooth outside $V(\Ff )\cup \mbox{Sing}(\Ff)$ and a subscheme of codimension $\ge 4$ (\cite{Chang} and Part (F) at page 4 in \cite{am}). If $n\ge 3$, we get
that $C$ is reduced because it has pure dimension $n-2$ and no embedded component.
\end{proof}

When $n=3$ and certain cohomology vanishing is true, then the converse also holds. In other words, given a locally Cohen-Macaulay curve $C\subset \PP^3$
and $r-1$ sections of $\omega _C(c_1(\PP^3)-c_1)\simeq \omega_C(4-c_1)$ spanning it, then we get a unique reflexive sheaf fitting in an exact sequence
(\ref{eqa1}) (\cite{Hartshorne1}, Theorem 4.1 and its generalizations to higher ranks in \cite{Arrondo}).

\begin{remark}
Let $\Ff$ be any sheaf fitting into the sequence (\ref{eqa1}). Then $\Ff$ has no torsion and $C\supseteq \mbox{Sing}(\Ff)$. If $\Ff$ is spanned outside finitely many points, then $\Ii _C(c_1)$ is spanned outside finitely many
points. Since $h^1(\Oo _{\PP^n})=0$, we have $V(\Ff )=V(\Ii _C(c_1))$.
In other words, $\Ff$ is spanned outside an algebraic set $V$ if and only if $\Ii _C(c_1)$ is spanned outside $V$.
Although $C$ depends on the choice of $r-1$ sections of $\Ff$, the set $V(\Ii _C(c_1))$ depends only on $\Ff$ and hence it may be used
to check if two sheaves are isomorphic.
\end{remark}

\begin{lemma}\label{a2}
Let $\Ff$ be a coherent sheaf of rank $r$ on $\PP ^n$ with $c_1=c_1(\Ff)$. 
\begin{enumerate}
\item For a line $L\subset \PP ^n$ with $L\nsubseteq \mathrm{Sing}(\Ff )$, $\Ff \vert _L$ is the direct sum of a torsion sheaf $\tau$ and a vector bundle $\Ee$ on $L$ with $\mathrm{rank}(\Ee) = r$.
\item We have $c_1 = \deg (\Ee)+\deg (\tau )$ and $\deg (\tau ) \ge \sharp (L\cap  \mathrm{Sing}(\Ff ))$. If $L\nsubseteq V(\Ff )$, then we have $\deg (\Ee)\ge 0$.
\end{enumerate}
\end{lemma}

\begin{proof}

Note that every coherent sheaf on a smooth curve is isomorphic to the direct sum of its torsion part and a locally free sheaf. Hence $\Ff \vert_ L$ is the direct sum of a torsion sheaf $\tau$ and a vector bundle $\Ee$ with $\mbox{rank}(\Ee) = r$. 

Let us fix a general hyperplane $H\subset \PP^n$ containing $L$ such that $H\nsupseteq \mbox{Sing}(\Ff )$. Let $\chi _{\Ff}(t)$ and $\chi _{\Ff \vert_ H}(t)$ denote the Hilbert polynomials of $\Ff$ and $\Ff \vert_ H$, respectively, and then from the exact sequence
\begin{equation}\label{eqa01}
0 \to \Ff (t-1) \to \Ff (t) \to \Ff (t)\vert _H\to 0
\end{equation}
we get $\chi _{\Ff \vert _H}(t) = \chi _{\Ff}(t) -\chi _{\Ff}(t-1)$. Since $L\nsubseteq \mbox{Sing}(\Ff )$, we have $\dim (H\cap \mbox{Sing}(\Ff ))=\dim (\mbox{Sing}(\Ff))-1$. If $n=2$, the we get $\chi _{\Ff}(t ) = rt +r+c_1$. If $n>2$ we get
$r =\mbox{rank}(\Ff \vert_ H)$ and $c_1 = c_1(\Ff \vert _H)$. Then we use induction on $n$ and the Riemann-Roch formula on $L$ to
get $\chi (\Ff \vert_ L) =rt+r+c_1$. 

For a point $P\in \mbox{Sing}(\Ff )\cap L$, let $\Ff _P$ (resp. $\Ff_{P,L}$) be the germ of the sheaf
$\Ff$ (resp. $\Ff \vert_ L$) at $P$. We also let $\mathfrak{m}$ (resp. $\mathfrak{m}_L$) be the maximal ideal of the local ring $\Oo _{\PP^n,P}$ (resp. $\Oo _{L,P}$). Since the tensor product is a right exact functor, we get that the vector space
$\Ff _P/{\mathfrak{m}}\Ff _P$ has dimension $>r$ and that this is the dimension of the vector space $\Ff_{P,L} /\mathfrak{m}_L\Ff_{P,L}$. The latter dimension is greater than $r$ if and
only if $P$ is in the support of $\tau$.  If $L\nsubseteq V(\Ff )$, then $\Ee$ is generically spanned and hence $\deg (\Ee)\ge 0$. Hence we get the second part of the lemma.
\end{proof}

\begin{remark}\label{a0}
Assume $c_1>0$ and so $\det (\Ff )$ is not trivial. Then the rank of any map $\Oo _X^{\oplus r}\to \Ff$ drops on a set of codimension at least $1$.
Thus if $V(\Ff )$ has codimension at least 2, then we have $h^0(\Ii _C(c_1)) \ge 2$.
\end{remark}

\begin{lemma}\label{a3}
Let $\Ee$ be a vector bundle of rank $r$ on $\PP^n$ with $c_1(\Ee)=c_1$ such that $V(\Ee )\ne \PP^n$. For a line $L\nsubseteq V(\Ee )$, we have $c_1 \ge \sharp (L\cap V(\Ff ))$.
\end{lemma}

\begin{proof}
Let $V\subseteq H^0(L,\Ee \vert _L)$ denote the image of the restriction map $H^0(\Ee )\to H^0(L,\Ee \vert_ L)$. Use that $\Ee \vert _L$ has degree $c_1$ and that
the cokernel of the evaluation map $V\otimes \Oo _L \to \Ee \vert_ L$ has cokernel supported at the points $L\cap V(\Ff )$. 
\end{proof}

%%%%%%%%%%%%%%%%%%%%%%%%%%%%%%%%%%%%%%%%

\section{Reflexive sheaves with $c_1=1$}\label{S1}

Let us fix two integers $n$ and $k$ such that $2 \le k \le n$. From $k+1$ linearly independent forms $u_1, \dots, u_{k+1}$ on $\PP^n$, we can construct an injective map of sheaves $j:\Oo_{\PP^n}(-1) \to \Oo_{\PP^n}^{\oplus (k+1)}$. Let $L$ be an $(n-k-1)$-dimensional linear subspace of $\PP^n$ defined by $\{u_1=\cdots = u_{k+1}=0\}$ and then the cokernel of $j$ only depends on the choice of $L$, i.e. it does not depend on the choice of the forms $u_1, \dots, u_{k+1}$ defining $L$.  Thus we can denote it by $\Ff_{k,n,L}$.
\begin{equation}
\Ff_{k,n,L}:=\mbox{coker}(~j:\Oo_{\PP^n}(-1) \to \Oo_{\PP^n}^{\oplus (k+1)}~).
\end{equation}

For simplicity, we often write $\Ff_{n-1,n,P}$ instead of $\Ff_{n-1,n,\{P\}}$ if $L=\{P\}$ is a single point space. 

\begin{remark}\label{oo6}
Let $L'\subset \PP^n$ be another linear subspace of $\PP^n$ with codimension $k$. Then there is a map $\phi\in \mbox{Aut}(\PP^n)$ such that
$\phi(L) =L'$ and thus we have $\phi^\ast (\Ff _{k,n,L'}) \cong \Ff _{k,n,L}$, i.e. $\Ff _{k,n,L}$ and $\Ff _{k,n,L'}$ are projectively equivalent.
\end{remark}

As an example, we have $\Ff _{n,n,\emptyset } = T\PP^n(-1)$ from the Euler sequence. By definition, $\Ff_{k,n,L}$ is a spanned sheaf of rank $k$. Since $h^0(\Ff _{k,n,L} )=k+1$, we have $h^0(\Ff _{k,n,L}\oplus \Oo _{\PP^n}^{\oplus (r-k)} )=r+1$ for all $r\ge k$.

\begin{lemma}\label{oo7}
The sheaf $\Ff _{k,n,L}$ has no torsion and $\mathrm{Sing}(\Ff _{k,n,L})=L$. Moreover $\Ff _{k,n,L}$ is reflexive if and only if $k\ge 2$.
\end{lemma}

\begin{proof}
At each point $P\in \PP^n\setminus L$, the map $j$ defining $\Ff _{k,n,L}$ is a vector bundle embedding and so we have $\mbox{Sing}(\Ff _{k,n,L}) \subseteq L$
and $\Ff _{k,n,L}$ has rank $k$. Conversely, the rank of the map $j$ drops at $P\in L$ and so $\Ff_{k,n,L}$ is not locally free at $P$. Thus we have $\mbox{Sing}(\Ff_{k,n,L})=L$.

In particular, $\Ff _{n,n,\emptyset }$ is locally free and we have $\Ff _{n,n,\emptyset} \cong T\PP^n(-1)$ from the Euler sequence of $\PP^n$.

Let $\Tt$ be the torsion of $\Ff _{k,n,L}$ and let $\pi :  \Oo _{\PP^n}^{\oplus (k+1)} \to \Ff _{k,n,L}/\Tt$ denote the obvious quotient map. Since $\Oo _{\PP^n}^{\oplus (k+1)}$ is locally
free and $ \Ff _{k,n,L}/\Tt$ has no torsion, so the sheaf $\mbox{ker}(\pi )$ is reflexive (\cite{Hartshorne1}, Corollary 1.5). Since
$\mbox{ker}(\pi )$ has rank $1$, it is a line bundle (\cite{Hartshorne1}, Proposition 1.9). Since $\Tt$ is supported in $L$, the inclusion $\Oo _{\PP^n}
\to \mbox{ker}(\pi )$ is the identity map outside a subspace of codimension $\ge 2$. Thus this inclusion is the identity map everywhere and in particular we have $\Tt =0$, i.e. $\Ff _{k,n,L}$
has no torsion.

Thus if $\Ff_{k,n,L}$ is reflexive, then we have $k\geq 2$ (\cite{OSS}, Corollary II.1.1.10). Conversely, let us assume that $k\geq 2$. Since $\Ff _{k,n,L}$ has no torsion, the natural map
$i: \Ff _{k,n,L} \to \Ff _{k,n,L}^{\vee \vee }$ is injective. Set $\epsilon := \mbox{coker}(i)$ and then $\Ff _{k,n,L}$ is reflexive if and only if $\epsilon =0$. Assume
$\epsilon \ne \emptyset$. Since $\mbox{Sing}(\Ff _{k,n,L}) =L$, the support $\mbox{Supp}(\epsilon )$ of $\epsilon $ is contained in $L$.
Since $\phi^\ast (\Ff _{k,n,L}) \cong \Ff _{k,n,L}$ for every $\phi\in \mbox{Aut}(\PP^n)$ such that $\phi(L) =L$, we get
$\mbox{Supp}(\epsilon )=L$.  The reflexive sheaf $\Ff _{k,n,L}^{\vee \vee}$ is spanned outside $L$, i.e. outside a subset of codimension at least 3. By Proposition \ref{oo1}, we have either $\Ff _{k,n,L}^{\vee \vee} \cong \Oo _{\PP^n}(1) \oplus \Oo _{\PP^n}^{\oplus (k-1)}$ or $\Ff _{k,n,L}^{\vee \vee} \cong \Ff _{l,n,M}\oplus \Oo _{\PP^n}^{\oplus (k-l)}$ for some $l \in \{3,\dots ,k\}$
and some linear subspace $M\subset \PP^n$ with codimension $k+1$.

 First assume  $\Ff _{k,n,L}^{\vee \vee} \cong \Ff _{l,n,M}\oplus \Oo _{\PP^n}^{\oplus (k-l)}$. We get
that $\Ff _{k,n,L}^{\vee \vee}$ is spanned and $h^0(\Ff _{k,n,L}^{\vee \vee})=k+1$. Since $h^0(\Ff _{k,n,L})=k+1$ and $\Ff _{k,n,L}$ is a proper subsheaf of $\Ff _{k,n,L}^{\vee \vee}$, we get
a contradiction.

Now assume that $\Ff _{k,n,L}^{\vee \vee} \cong \Oo _{\PP^n}(1) \oplus \Oo _{\PP^n}^{\oplus (k-1)}$ and then we have an inclusion
$i_1: \Ff _{k,n,L}\hookrightarrow \Oo _{\PP^n}(1) \oplus \Oo _{\PP^n}^{\oplus (k-1)}$ with $\epsilon$ isomorphic to $\mbox{coker}(i_1)$. Since $\Ff _{k,n,L}$
is spanned, it implies that $\Ff _{k,n,L} \cong \Jj \oplus  \Oo _{\PP^n}^{\oplus (k-1)}$ with $L$ as the support of $ \Oo _{\PP^n}^{\oplus (k-1)}/\Jj$. Since $h^0(\Ff _{k,n,L})=k+1$, we get
$\Jj \cong \Ii _L(1)$ and so $\Ff _{k,n,L}\cong \Ii _L(1)\oplus \Oo _{\PP^n}^{\oplus (k-1)}$. For a fixed point $P\in L$, let ${\mathfrak{m}_P}$ be the maximal ideal of the local ring
$\Oo _{\PP^n,P}$. Since the tensor product is a right exact functor and $\Ff _{k,n,L}$ is a quotient of $\Oo _{\PP ^n}^{\oplus (k+1)}$, the vector space $\Ff _{k,n,L}/{\mathfrak {m}_P}\Ff _{k,n,L}$
has dimension $k+1$.
But the vector space $(\Oo _{\PP^n}(1) \oplus \Oo _{\PP^n}^{\oplus (k-1)})/{\mathfrak {m}_P}(\Oo _{\PP^n}(1) \oplus \Oo _{\PP^n}^{\oplus (k-1)})$ has dimension $(n+1) +(k-1) >k+1$ and it is a contradiction.
\end{proof}

\begin{proposition}\label{oo1}
A reflexive sheaf $\Ff$ of rank $r\ge 3$ on $\PP^n$, $n\ge3$ with $c_1=1$ has $\codim V(\Ff ) \ge 3$ if and only if $\Ff$ is isomorphic to either
\begin{enumerate}
\item $\Oo_{\PP^n}(1)\oplus \Oo_{\PP^n}^{\oplus (r-1)}$, or
\item $\Ff_{k,n,L}\oplus \Oo_{\PP^n}^{\oplus (r-k)}$ for some integer $k$ such that $2 \le k \le \min \{n,r\}$ and an $(n-k-1)$-dimensional linear subspace
$L\subset \PP^n$.
\end{enumerate}
\end{proposition}
\begin{proof}
Let us take $r-1$ general sections of $\Ff$ to have an exact sequence (\ref{eqa1}). If $C=\emptyset$, we have $\Ff\cong \Oo_{\PP^n}(1)\oplus \Oo_{\PP^n}^{\oplus (r-1)}$. 

Let us assume that $C\ne \emptyset$ and let $M$ denote the linear span of $C$ in $\PP^n$, the intersection of all linear subspaces containing the scheme $C$. By convention, we set $M = \PP^n$ if there is no such a linear subspaces. Since $C$ has pure codimension $2$ by Lemma \ref{oo4} and $\Ii _C(1)$ is spanned outside
$V(\Ff )$, we have $h^0(\Ii _C(1)) \ge 2$. Since $H^0(\Ii _C(1))= H^0(\Ii _M(1))$, we get $\dim (M)\le n-2$. So the only possibility is that $\dim (M)=n-2$ and $C=M$. Thus the scheme $C$ is a linear subspace of codimension 2 and $h^0(\Ff)=r+1$. 

Since $\Ii_C(1)$ is spanned
and $h^1(\Oo _{\PP^n})=0$, so $\Ff$ is spanned. Since $\Ff$ is torsion-free and $\Oo_{\PP^n}^{\oplus (r+1)}$ is locally free, the sheaf $\mbox{ker}(e_{\Ff })$ is a reflexive sheaf of rank 1
(\cite{Hartshorne1}, Proposition 1.1). Thus $\mbox{ker}(e_{\Ff })$ is a line bundle (\cite{Hartshorne1}, Proposition 1.9) and $\Ff$ fits into an exact sequence
\begin{equation}\label{eqa2}
0 \to \Oo _{\PP ^n}(-1) \stackrel{\phi}{\to} \Oo _{\PP ^n}^{\oplus( r+1)} \to \Ff \to 0
\end{equation}
with $\phi = (u_1,\dots ,u_{r+1})$ and each $u_i\in H^0(\Oo _{\PP ^n}(1))$. Up to an automorphism of $\Oo _{\PP ^n}^{\oplus( r+1)}$, we may assume that $u_1,\dots ,u_{k+1}$
are linearly independent and $u_i=0$ for all $i>k+1$. Hence $k$ is a positive integer with $k \le \min \{n,r\}$. Hence $\Oo _{\PP^n}^{\oplus (r-k)}$ is a direct factor
of $\Ff$, i.e. $\Ff\cong \Oo_{\PP^n}^{\oplus (r-k)}\oplus \Ff'$ for some reflexive sheaf $\Ff'$ of rank $k$. Let $L$ be the common zero-locus of $u_1, \dots, u_{k+1}$ which is an $(n-k-1)$-dimensional subspace of $\PP^n$. Then we have $V(\Ff)=L$ and $\Ff' \cong \Ff_{k,n,L}$. Now by Lemma \ref{oo7}, $\Ff$ is reflexive if and only if $k\ge2$. 
\end{proof}

As an automatic consequence of this proof, we get that $\Ff$ is locally free if and only if $\Ff\cong T\PP^n (-1)\oplus \Oo_{\PP^n}^{\oplus (r-n)}$ with $r\geq n$.  

\begin{remark}\label{a12}
About the cohomology groups of $\Ff_{k,n,L}$, we can observe the followings:
\begin{enumerate}
\item From (\ref{eqa2}) we get $h^i(\Ff _{k,n,L}(t)) =0$ for all $t\in \mathbb {Z}$ and all $i\in \{1,\dots ,n-2\}$ and $h^{n-1}(\Ff _{k,n,L}(t)) = 0$ for all $t\ge -n+1$.
\item Fix integers $ n, k$ such that $n\ge 3$ and $2 \le k \le n-1$. Let $L\subset \PP^n$ be an $(n-k-1)$-dimensional linear subspace of $\PP^n$ and let $u_1,\dots, u_{k+1}$ be linear forms such that $L = \{u_1=\cdots =u_{k+1}=0\}$. Then there is a family
of maps $\phi _\lambda
: \Oo _{\PP^n}(-1) \to \Oo _{\PP^n}^{\oplus (n+1)}$, $\lambda \in \Delta$, $\Delta$ a smooth and connected affine curve, and $o\in \Delta$ such that $\phi _o = (u_1,\dots ,u_{k+1},0,\dots ,0)$,
while $\phi _\lambda$ is induced by the choice of a basis of $H^0(\Oo _{\PP^n}(1))$ for all $\lambda \in \Delta \setminus \{o\}$.
We get that $\Ff _{k,n,L}\oplus \Oo _{\PP ^n}^{\oplus (n-k)}$ is the flat limit of a family of vector bundles isomorphic to $T\PP ^n(-1)$. Hence $\Ff _{k,n,L}$ and $T\PP ^n(-1)$ have the same
Segre classes. By the semicontinuity theorem for cohomology (\cite{Hartshorne}, III.12.8) we get $h^{n-1}(\Ff _{k,n,L}(-n)) >0$. Since (\ref{eqa2}) gives $h^{n-1}(\Ff _{k,n,L}(-n)) \le 1$,
we get $h^{n-1}(\Ff _{k,n,L}(-n)) =1$.
\end{enumerate}
\end{remark}

As an example, let us take a look at $\Ff_{2,3,P}$ with $P$ a point in $\PP^3$ for our later use. The sheaf $\Ff_{2,3,P}$ is a spanned sheaf of rank 2 with $h^0(\Ff _{2,3,P}(-1)) =0$ and $h^0(\Ff _{2,3,P}) = 3$. By Remark 2.5.1 in \cite{Hartshorne1}, we have $c_3(\Ff _{2,3,P}) = 1$.

\begin{lemma}\label{lemma}
Some cohomological properties of $\Ff_{2,3,P}$ are as follows:
\begin{enumerate}
\item $h^0(\mathcal{E}xt^1(\Ff _{2,3,P},\Oo _{\PP^3}(r))=1$ for each $r\in \mathbb {Z}$.
\item $h^1(\Ff _{2,3,P}(t))=h^1(\Ff _{2,3,P}^\vee (t))=0$ for all $t\in \mathbb {Z}$.
\item $h^2(\Ff _{2,3,P}(t)) \le 1$ and $\dim (\mathrm{Ext}^1(\Ff _{2,3,P},\Oo_{\PP^3}(t))) \le 1$ for all $t\in \ZZ$. 
\item $h^2(\Ff _{2,3,P}( -4)) = \dim (\mathrm{Ext}^1(\Ff _{2,3,P},\Oo_{\PP^3})) = 1$.
\item $\dim (\mathrm{Ext}^1(\Ff _{2,3,P},T\PP^3(-1)) \ge 3$.
\end{enumerate}
\end{lemma}

\begin{proof}
\quad{(1)} Note that $\mathcal{E}xt^1(\Ff _{2,3,P},\Oo _{\PP^3})$ is a skyscraper sheaf supported by $P$ and thus we have $$h^0(\mathcal{E}xt^1(\Ff _{2,3,P},\Oo _{\PP^3})(r))
= h^0(\mathcal{E}xt^1(\Ff _{2,3,P},\Oo _{\PP^3})(s))$$ for all $r,s\in \mathbb {Z}$. The number is $c_3(\Ff _{2,3,P}) =1$ by Proposition 2.6 in \cite{Hartshorne1}.

\quad{(2)} Now from the Euler sequence
$$0 \to \Oo_{\PP^3}(-1) \to \Oo_{\PP^3}^{\oplus 4} \to T\PP^3(-1)\to 0$$
we get $h^1(\Ff _{2,3,P}(t))=0$ for all $t\in \mathbb {Z}$. Since $\Ff _{2,3,P}^\vee \cong \Ff _{2,3,P}(-1)$ (\cite{Hartshorne1}, Proposition 1.10), we get $h^1(\Ff _{2,3,P}^\vee (t))=0$ for all $t\in \mathbb {Z}$. 

\quad{(3),(4)} The Serre duality gives $h^2(\Ff _{2,3,P}(t))  =\dim (\mbox{Ext}^1(\Ff _{2,3,P},\Oo_{\PP^3}(-4-t)))$ for all $t\in \mathbb {Z}$. Since $h^0(\mathcal{E}xt^1(\Ff _{2,3,P},\Oo _{\PP^3})(r))=1$ for each $r\in \mathbb {Z}$, the last part of Theorem 2.5 in \cite{Hartshorne}, gives $h^2(\Ff _{2,3,P}(t)) \le 1$ and that $h^2(\Ff _{2,3,P}(t)) = 1$ if and only
if $h^2(\Ff _{2,3,P}^\vee (-t-4)) =0$, i.e. if and only
if $h^2(\Ff _{2,3,P}(-t-5)) =0$. Taking $t=-4$ we get $h^2(\Ff _{2,3,P}( -4)) =1$ by Remark \ref{a12}.

\quad{(5)} Taking $\mathcal{H}om (\Ff_{2,3,P}, \cdot )$ functor to the Euler sequence, we obtain the exact sequence
$$\mbox{Ext}^1(\Ff _{2,3,P},\Oo_{\PP^3}(-1)) \to \mbox{Ext}^1(\Ff _{2,3,P},\Oo_{\PP^3}^{\oplus 4}) \to \mbox{Ext}^1(\Ff _{2,3,P},T\PP^3(-1)).$$
Use $\dim (\mbox{Ext}^1(\Ff _{2,3,P},\Oo_{\PP^3}(-1))) \le 1$ and $\dim (\mbox{Ext}^1(\Ff _{2,3,P},\Oo_{\PP^3})) = 1$.
\end{proof}

\begin{remark}\label{a14}
Fix an integer $r\in \{2,\dots ,n-1\}$ and any $\Ff=\Ff _{n-1,n,P}$ where $P$ is a point in $\PP^n$. Let $R\subset \PP^n$ be a line. If $P\notin R$, then $\Ff \vert_ R$ is a spanned vector bundle of rank $n-1$ with degree $1$
and thus $\Ff \vert _R$ has splitting type $(1,0,\cdots ,0)$. If $P\in R$, then $\Ff \vert _R$ is a direct sum of a non-zero torsion sheaf $\tau$ supported by $P$ and a spanned rank $r$ vector bundle
$\Ee$ on $R$ with $\deg (\Ee )=1-\deg (\tau )$ by Lemma \ref{a2}. Hence $\deg (\tau )=1$, i.e. $\tau = \Oo _P$ and $\Ee \cong \Oo _R^{\oplus r}$. 
In general, let us consider an arbitrary
$\Ff=\Ff _{k,n,L}$. If $R\subset \PP^n$ is a line such that $R\cap L=\emptyset$, then the vector bundle $\Ff \vert _R$ has splitting type $(1,0,\dots ,0)$.
If $R\cap L$ is a point, say $P$, then $\Ff \vert_R$ is isomorphic to the direct sum of $\Oo _P$ and a trivial vector bundle $\Oo_R^{\oplus r}$.
\end{remark}

%%%%%%%%%%%%%%%%%%%%%%%%%%%%%%%%%%%%

\section{Reflexive sheaves on $\PP^3$ with $c_1=2$}
The main goal of this section is to prove the following result.

\begin{proposition}\label{OOO}
Let $\Ff$ be an indecomposable and reflexive sheaf of rank $r$ on $\PP^3$ with $c_1(\Ff)=2$ and $\dim V(\Ff)\le 0$. Then we have $r\le 8 $, $c_2(\Ff)\le 4$ and $\sharp (\mathrm{Sing}(\Ff))\le 8$. 
\end{proposition} 

In fact we investigate each cases in more detail and give complete description in some cases. 

\begin{remark}\label{lem}
For a non-locally free, reflexive and decomposable sheaf $\Ff$ without trivial factor on $\PP^3$ with $c_1=2$, we have $\sharp (\mathrm{Sing}(\Ff))\le 2$.  Indeed, by Proposition \ref{oo1}, $\Ff$ is isomorphic to either $\Ff_{2,3,P}\oplus \Oo _{\PP^3}(1)$, $\Ff _{2,3,P}\oplus T\PP^3(-1)$ or $\Ff _{2,3,P}\oplus \Ff _{2,3,P'}$ for some $P,P'\in \PP^3$ (possibly $P=P'$). In each case, we have $\sharp (\mbox{Sing}(\Ff))\le 2$ and only when $\Ff\cong \Ff_{2,3,P}\oplus \Ff_{2,3,P'}$ with $P\ne P'$, we have $\sharp(\mbox{Sing}(\Ff))=2$.

 In other words, if $\sharp(\mbox{Sing}(\Ff))\ge 3$, then $\Ff$ is indecomposable. 
\end{remark}

Let $\Ff$ be a reflexive sheaf of rank $r$ on $\PP^3$ with $c_1(\Ff)=2$ and $ \dim V(\Ff)\le 0$. Then it admits an exact sequence (\ref{eqa1}) and $C$ is either empty or locally Cohen-Macaulay with $\mbox{Sing}(C) \subseteq V(\Ff )\cup \mbox{Sing}(\Ff)$ such that $\Ii _C(2)$ is spanned outside $V(\Ff )$. Since $\mbox{Sing}(C)$ is finite and $C$  is locally Cohen-Macaulay, so $C$ is reduced. 

\begin{remark}\label{b1}
Fix a linear subspace $V\subseteq H^0(C,\omega _C(2))$ spanning $\omega _C(2)$ outside finitely many points, but not spanning $\omega _C(2)$ at exactly $\alpha >0$
points $P_1,\dots ,P_\alpha$.
Then the reflexive sheaf $\Ff$ associated to (\ref{eqa1}) is not locally free and $\mbox{Sing}(\Ff )= \{P_1,\dots ,P_\alpha \}$.
\end{remark}

The extension (\ref{eqa1}) is given by $r-1$ sections of $\omega _C(2)$ and we set $V\subseteq H^0(C,\omega _C(2))$ to be their linear span. By our assumptions and Remark \ref{b1}, $V$ spans $\omega _C(2)$ outside $\mbox{Sing}(\Ff )$. Since $\Ii _C(2)$ is spanned outside
the finite set $V(\Ff )$, then we have $h^0(\Ii _C(2)) \ge 2$ and so $c_2(\Ff)=\deg (C)\le 4$. 

Moreover, there is no line $L$ with $\deg (L\cap C)\ge 3$, which is not a component of $C$. Note that we have $C=\emptyset$
if and only if $\Ff \cong \Oo _{\PP^3}(2)\oplus \Oo _{\PP ^3}^{\oplus (r-1)}$. Thus let us assume $C\ne \emptyset$ from now on. 

\begin{lemma}
If $\deg (C)=1$, then we have $\Ff \cong \Oo_{\PP^3}(1)^{\oplus 2} \oplus \Oo_{\PP^3}^{\oplus (r-2)}$. 
\end{lemma}
\begin{proof}
 Note that $C$ is a line. Since $\Ii _C(2)$ is spanned and $h^1(\Oo _{\PP^3})=0$, so $\Ff$ is spanned. Any such $\Ff$ is given by
$r-1$ sections $s_1,\dots ,s_{r-1}$ of $\omega _C(2) \cong \mathcal {O}_C$ with only finitely many common zeros. Since $h^0(C,\mathcal {O}_C)=1$,
by an automorphism of $\Oo _{\PP^3}^{\oplus (r-1)}$ we can assume that one of them is nowhere vanishing and the others are the zero-section. In other words, $\Ff$ is locally free and it has $\Oo_{\PP^3}^{\oplus (r-2)}$ as its direct factor. So we can assume $\mbox{rank}(\Ff)=2$.  From the main theorem in \cite{SU}, the only spanned vector bundle $\Ff$ of rank $2$ with $h^0(\Ff(-1))=h^0(\Ii_C(1))=2$ is $\Oo _{\PP^3}(1)^{\oplus 2}$.
\end{proof}

Now assume that $c_2(\Ff)=\deg (C)=2$. Then $C$ is either the disjoint union of two lines or a reduced conic. 

\begin{proposition}
Let $\Ff$ be a reflexive sheaf of rank $r$ on $\PP^3$ with $(c_1, c_2)=(2,2)$ and $\dim V(\Ff)\le 0$. Then $\Ff$ is one of the followings:
\begin{enumerate}
\item $\Ff\cong \Nn_{\PP^3}(1)\oplus \Oo_{\PP^3}^{\oplus (r-2)}$ or $\Ff \cong \Omega_{\PP^3}(2)\oplus \Oo_{\PP^3}^{\oplus (r-3)}$.
\item $0\to \Oo_{\PP^3} \to \Oo_{\PP^3}(1)\oplus T\PP^3(-1) \oplus \Oo_{\PP^3}^{\oplus (r-3)} \to \Ff \to 0$.
\item $\Ff \cong \Gg \oplus \Oo_{\PP^3}^{\oplus (r-2)}$ with $\Gg \in \mathfrak{A}$, a 6-dimensional irreducible variety. 
\end{enumerate}
\end{proposition}

\begin{proof}
When $C$ is the disjoint union of two lines, say $C = L_1\sqcup L_2$, we have $\omega _C(2) \cong \Oo _C^{\oplus 2}$ and $\Ii _C(2)$
is spanned. In particular $\Ff$ is also spanned, i.e. $V(\Ff )=\emptyset$. Note that $h^0(C,\omega _C(2))=2$ and thus $\Oo _{\PP ^3}^{\oplus (r-3)}$ is a direct factor of $\Ff$ if $r\ge 4$.
Hence it is sufficient to analyze the cases $r\in \{2,3\}$. If $r=2$, then the sequence (\ref{eqa1}) is induced by a nowhere zero section of $\omega _C(2)$ and so $\Ff$
is locally free. Thus $\Ff$ is isomorphic to $\Nn_{\PP^3}(1)$, a null-correlation bundle of $\PP^3$ twisted by $1$ from the list in \cite{SU}.
Now assume that $r=3$. If $P$ is a point in $ L_i$, then any section of $\omega _C(2)$ vanishing at $P$ also vanishes at every point of $L_i$. It implies again that $\Ff$ is locally free (see Theorem 1 in \cite{Arrondo}). Assuming that $\Ff$ has no trivial factor, we have $\Ff \cong \Omega _{\PP ^3}(2)$ (e.g., by \cite{SU}). 

 Now assume that $C$ is a reduced conic. Since $\Ii_C(2)$ is spanned, so is $\Ff$. Note that $h^0(C,\omega_C(2))=3$ and thus $\Ff$ has $\Oo_{\PP^3}^{\oplus (r-4)}$ as its direct factor if $r\geq 5$. So it is sufficient to analyze the cases $2\leq r\leq 4$. The case of locally free sheaves turns out to satisfy the resolution (2) in \cite{SU}, so let us assume that $\Ff$ is non-locally free. Since $\omega_C(2)$ is spanned by 3 linearly independent sections of $H^0(\omega_C(2))$, we have $r\leq 3$. If $r=3$, we have $\Ff \cong \Ff_{2,3,P}\oplus \Oo_{\PP^3}(1)$ for some $P\in \PP^3$ by Proposition \ref{a19}. From the defining sequence of $\Ff_{2,3,P}$ and the Euler sequence of $T\PP^3(-1)$, we have a surjection $T\PP^3(-1) \to \Ff_{2,3,P}$. Thus $\Ff$ also admits a resolution (2).  
 
The remaining case of $r=2$ is induced from Proposition \ref{c5} giving us the last case in the list.  
\end{proof}

In particular, such a sheaf with no trivial factor has $r\le 3$, $c_3(\Ff) \in \{0,2\}$ and $\sharp (\mathrm{Sing}(\Ff))\le 2$. 

\begin{lemma}\label{c0}
Let $\Ff$ be a reflexive sheaf of rank 2 on $\PP^3$ fitting into an exact sequence
\begin{equation}\label{eqa1.1}
0 \to \Oo _{\PP^3}\to \Ff  \to \Ii _C(2) \to 0
\end{equation}
with $C$ a reduced conic. Then there is an effective Cartier divisor $Z\subset C$ of degree $2$ such that $\mathcal {O}_C(1) \cong \mathcal {O}_C(Z)$,
where $Z$ is the scheme-theoretic zero-locus of a section of $\Oo _C(1)$ with only finite zeros
and $\Oo _Z \cong \mathcal{E}xt^1(\Ff ^{\vee },\Oo _{\PP^3}(-2))$. The scheme $Z$ is uniquely determined by $\Ff$.

If $C$ is reducible, say $C =L_1\cup L_2$ with $L_1$ and $L_2$ lines, then we have $\deg (Z\cap L_1) = \deg (Z\cap L_2)=1$.
\end{lemma}

\begin{proof}
We have $\omega _C(2) \cong \Oo _C(1)$. As in \cite{Hartshorne1}, Theorem 4.1, $\Ff$ is uniquely determined by
$s\in H^0(C,\Oo _C(1))$ which vanishes only at finitely many point, i.e. it does not vanish identically in one of the irreducible components of $C$.
Hence the scheme-theoretic zero-locus of $s$ is a degree $2$ effective divisor $Z$ of $C$. If $C$ is reducible, say $C =L_1\cup L_2$ with $L_1$ and $L_2$ lines, then we have $\deg (Z\cap L_1) = \deg (Z\cap L_2)=1$. As in the equation (2) in the proof of Theorem 4.1 in \cite{Hartshorne1}, we have an exact sequence
\begin{align*}
&0 \to \Oo _{\PP^3}(-2) \to \Ff(-2) \stackrel{f}{\to} \Oo _{\PP^3}\stackrel{g}{\to} \mathcal{E}xt^1(\Ii _C,\Oo _{\PP^3}(-2))\\
&\to \mathcal{E}xt^1(\Ff ^{\vee },\Oo _{\PP^3}(-2)) \to 0.
\end{align*}
Here we have $\Ii _C = \mbox{Im}(f)$ and the sheaf $\mathcal{E}xt^1(\Ii _C,\Oo _{\PP^3}(-2))$ is identified with $\omega _C(2)$ (\cite{Hartshorne1}, proof of Theorem 4.1).
Thus we have $\Oo _Z \cong \Oo _{\PP^3}/\mbox{ker}(g ) \cong \mathcal{E}xt^1(\Ff ^{\vee },\Oo _{\PP^3}(-2))$.
\end{proof}

\begin{lemma}\label{c3}
Fix a reduced conic $C\subset \PP^3$ and a Cartier divisor $Z\subset C$ of degree 2. If $C$ is reducible, say $C =L_1\cup L_2$ with $L_1$ and $L_2$ lines, assume $\deg (Z\cap L_1) = \deg (Z\cap L_2)=1$. Then there is a unique reflexive sheaf $\Ff _{C,Z}$ of rank $2$ fitting into an extension (\ref{eqa1.1}) associated to a section of $\omega _C(2)\cong \Oo _C(1)$
with $Z$ as its zero-locus.
\end{lemma}

\begin{proof}
Since $\omega _C(2) \cong \Oo _C(1)$, the lemma is a particular case of Theorem 4.1. in \cite{Hartshorne1}.
\end{proof}

\begin{lemma}\label{c4}
For a fixed 0-dimensional subscheme $Z\subset \PP^3$ of degree $2$, let $\Cc_1 $ be the set of all reduced conics containing $Z$ and let $\Cc_2$ be the set of all $C\in \Cc_1$ such that $Z$ is the zero-locus of a section of $\Oo _C(1)$. 

\quad (a) Both $\Cc_1$ and $\Cc_2$ are irreducible varieties of dimension
5. 

\quad (b) For a fixed $C\in \Cc_1$, we have $C\in \Cc_2$ if and only if either $C$ is smooth or $C$ is reducible, say $C=L_1\cup L_2$ with $L_1$ and $L_2$ lines, and
$\deg (Z\cap L_1) = \deg (Z\cap L_2) =1$.
\end{lemma}

\begin{proof}
Note that $Z$ is contained in a unique line, say $L$. Thus we have $h^0(\Ii _Z(1)) =h^0(\Ii _L(1))=2$ and so part (a) is a simple dimension counting.

 If $C$ is smooth, then every 0-dimensional
subscheme of $C$ is a Cartier divisor and we have $h^0(C,\Oo_C(1)\otimes \Oo_C(-Z)) = h^0(\Oo _C)=1$. Now assume that $C=L_1\cup L_2$ is reducible with $\{O\} = L_1\cap L_2$.
Let $\Pi$ be the plane spanned by $L_1\cup L_2$. Since $h^0(L_i,\Oo _{L_i}(1)) =2$,
we need to have $\deg (Z\cap L_i)\le 1$ for all $i$. Hence if $Z$ is reduced, then we have $\sharp (Z\cap (L_i\setminus \{O\})) =1$, $i=1,2$, while if $Z$ is not reduced,
then $Z_{red} =\{O\}$. Every 0-dimensional subscheme $W \subset \Pi$ of degree $2$ such that $W_{red}=\{O\}$ spans a unique line $L_W$. Any such
$W$ is contained in $L_1\cup L_2$. $W$ is a Cartier of $L_1\cup L_2$ if and only if $L\ne L_1$ and $L\ne L_2$, i.e. $\deg (W\cap L_1) =\deg (W\cap L_2)=1$.
\end{proof}

\begin{proposition}\label{c5}
Let $\mathfrak{A}$ be the set of all spanned reflexive sheaves of rank 2 on $\PP ^3$ with $(c_1,c_2,c_3) = (2,2,2)$. Then $\mathfrak{A}$ is parametrized by an irreducible variety of dimension
$6$.
\end{proposition}

\begin{proof}
The set of all 0-dimensional subschemes
$Z\subset \PP^3$ of degree 2 is parametrized by an irreducible, smooth and projective variety of dimension 6. Note that $\mathfrak{A}$ is the set of all reflexive sheaves $\Ff$ fitting into an exact sequence (\ref{eqa1.1}) with $C$ a reduced conic. Since $\omega _C(2) \cong \Oo_C(1)$ is
very ample, we have $\mathfrak{A}\ne \emptyset$ by \cite{Hartshorne1}, Theorem 4.1.

For a fixed $\Ff\in \mathfrak{A}$, we have $h^0(\Ff )=6$ and so the non-zero sections of $\Ff$ are
parametrized by a 5-dimensional projective space. Lemma \ref{c0} shows that $ \mathcal{E}xt^1(\Ff ^{\vee },\Oo _{\PP^3}(-2)) \cong \Oo_Z$ for a unique subscheme $Z$ of degree $2$. Lemmas \ref{c0} and \ref{c4} show
that any reduced conic $C$ associated to a non-zero section of $\Ff$ contains $Z$ as a Cartier divisor and that if $C$ is reducible, then no irreducible component
of $C$ contains $Z$. Lemma \ref{c3} says that any pair $(C,Z)$ as above is associated to a unique $\Ff$, up to isomorphisms. Since $h^0(\Ff )=6$, we see
that every conic $C$ containing the scheme $Z$ associated to $ \mathcal{E}xt^1(\Ff ^{\vee },\Oo _{\PP^3}(-2))$ arises from a non-zero section
of $\Ff$.
\end{proof}

\begin{proposition}\label{a19}
Let $\Ff$ be a non-locally free and reflexive sheaf of rank 3 on $\PP^3$ whose associated curve $C$ is a reduced conic with $\dim V(\Ff)\le 0$. Then we have $\Ff \cong \Ff _{2,3,P}\oplus \Oo _{\PP^3}(1)$ for some $P\in \PP^3$.
\end{proposition}

\begin{proof}
Let $\Delta _1$ denote the set of all pairs $(T,P)$, where $T\subset \PP^3$ is a reduced conic and $P\in T$. The set $\Delta _1$ is an irreducible quasi-projective variety.
For each $(T,P)\in \Delta _1$, we have $h^0(T,\Ii _P(1)) =2$ and $\vert \Ii _P(1)\vert$ has no base points. Since
$\omega _T(2) \cong \Oo _T(1)$ for any conic $T$, the set $\Delta$ of all non-locally free reflexive sheaves of rank 3 on $\PP^3$ admitting the sequence (\ref{eqa1}) with $C=T\in \Delta_1$, is parametrized by an irreducible
quasi-projective variety. Let us fix $\Ff\in \Delta$. Since we have $h^0(\Ff (-1)) =1$ and $h^0(\Ff (-2)) =0$, there is an exact sequence
\begin{equation}\label{eqc1}
0 \to \Oo _{\PP^3}(1) \stackrel{i}{\to} \Ff \to \Gg \to 0
\end{equation}
with $\Gg$ torsion-free, $\mbox{rank}(\Gg )=2$, $c_1(\Gg )=1$ and $h^0(\Gg )=3$. 

 Now assume that $\Gg$ is not reflexive. Let $i: \Gg \hookrightarrow \Gg ^{\vee \vee }$ denote the natural inclusion. First assume that
$\mbox{coker}(j)$
is finite. In this case $V(\Gg ^{\vee \vee })$ is finite. By Proposition \ref{OOO2} we get $\Gg ^{\vee \vee }\cong \Ff _{2,3,P'}$ for some $P'$. Since $h^0(\Ff _{2,3,P'})=3 =h^0(\Gg )$
and both $\Gg $ and $\Ff _{2,3,P'}$ are spanned, we get $\Gg \cong \Gg ^{\vee \vee }$, a contradiction. Now assume that $\mbox{coker}(i)$ is supported by a curve
$Y$, i.e. the one-dimensional part of $\mbox{Sing}(\Gg )$ is not empty. First assume that $Y$ is not a line. There is a line $R\subset \PP^3$
such that $\sharp (R\cap Y)\ge 2$ and $R$ is not contained in $Y$. Hence the torsion part of $\Gg \vert_ R$ is supported by at least two point, say $P_1$ and $P_2$.
At $P_i$ either $\Ff \vert_ R$ has torsion or the restriction of $i$ to $R$ is not a vector bundle inclusion. In the latter case $\Oo _R(2)$ is a subsheaf
of $\Ff\vert _R$ and hence $\Ff \vert _R$ has splitting type $(2,0,0)$. We excluded the latter case for any $R$, while the former case does not occur, because
$P_1\ne P_2$. Now assume $Y=R$. As above we get a contradiction if $\mbox{Sing}(\Gg )$ has an isolated point. Hence we may assume that $Y = \mbox{Sing}(\Gg )$
is a line. Since $\Gg$ is spanned, $h^0(\Gg )=3$ and $\Gg$ has no torsion, the kernel of the evaluation map $H^0(\Gg )\otimes \Oo _{\PP^3} \to \Gg$ is a line bundle.
Hence $\Gg$ fits into an exact sequence
\begin{equation}\label{eqc2}
0 \to \Oo _{\PP^3}(-1) \stackrel{v}{\to } \Oo _{\PP^3}^{\oplus 3}\to \Gg \to 0
\end{equation}
with $v = (v_1,v_2,v_3)\in H^0(\Oo_{\PP ^3}(1))^{\oplus 3}$. Since $\mbox{Sing}(\Gg)=Y$ is a line, the sections $\{v_1,v_2,v_3\}$ generate a 2-dimensional linear space
and $\Gg \cong \Ii_Y(1)\oplus \Oo _{\PP^3}$. Since $\Oo _{\PP^3}$ is a direct summand of $\Gg$ and $\Gg$ is a quotient of $\Ff$, so $\Oo_{\PP^3}$ is a direct summand of $\Ff$, a contradiction.
Thus $\Gg$ is reflexive and so $\Gg\cong \Ff_{2,3,P}$ for some $P\in \PP^3$. 

Now from the sequence (\ref{eqc1}) it is easy to check that $\Ff$ is not semistable if and only if the sequence (\ref{eqc1}) splits,
i.e. if and only if $\Ff \cong \Ff _{2,3,P}\oplus \Oo _{\PP^3}(1)$ for some $P\in \PP^3$. Hence non-split sheaves forms an open subset $U$ of $\Delta$ and we need
to prove that $U=\emptyset$. 

Assume that $U\ne \emptyset$. We just saw that $U$ is irreducible. Let us fix a general $\Ff \in U$. For any line $R\subset \PP^3$, the sheaf $\Ff _{2,3,P}\oplus \Oo _{\PP^3}(1) \vert_ R$ is a vector bundle of splitting type $(1,1,0)$
if $P\notin R$, while $\Ff _{2,3,P}\oplus \Oo _{\PP^3}(1) \vert_ R$ is a direct sum of $\Oo _P$ and a trivial vector bundle if $P\in R$ by Lemma \ref{a2}. Hence
there is no line with either $\Ff _{2,3,P}\oplus \Oo _{\PP^3}(1) \vert_ R$ of splitting type $(2,0,0)$ or with torsion of degree $\ge 2$. Since this condition is open, we get
that $\Ff$ satisfies this condition with respect to the point $\{P\}:= \mbox{Sing}(\Ff )$. Since $\Ff$ is spanned, so is $\Gg$. We get $\Gg \cong \Ff _{2,3,P}$
if $\Gg$ is reflexive. Since we have $\dim (\mbox{Ext}^1(\Ff _{2,3,P},\Oo _{\PP ^3}(1)))=1$ from Lemma \ref{lemma}, so there
is a unique sheaf $\Hh$, up to isomorphism, which is the middle term of a non-trivial extension of $\Ff _{2,3,P}$ by $\Oo _{\PP ^3}(1)$ and $\Hh$ is not isomorphic
to $\Ff _{2,3,P}\oplus \Oo _{\PP ^3}(1)$. To get a contradiction to the non-emptiness of $U$ it is sufficient to prove that $\Hh$ is locally free at $P$.  

We have $\Ff _{2,3,P} ^\vee \cong \Ff _{2,3,P}(-1)$ by Proposition 1.10 in \cite{Hartshorne1} and so we have $h^1(\Ff _{2,3,P}^\vee (1)) =0$. It implies that the non-zero extension giving $\Hh$
gives a non-zero element of $H^0(\mathcal{E}xt^1(\Ff _{2,3,P},\Oo _{\PP ^3}(1)))$. The latter vector space is 1-dimensional by Lemma \ref{lemma} and so it is sufficient to find a locally free extension
of $\Ff _{2,3,P}$ by $\Oo _{\PP ^3}(-1)$. Take a general inclusion $\Oo _{\PP ^3} \to T\PP^3(-1)$. Its cokernel $\Gamma$  is a reflexive spanned sheaf of rank $2$. By the list
we get that $\Gamma \cong \Ff _{2,3,P'}$ for some $P'\in \PP^3$. Composing with an automorphism of $\PP^3$ sending $P'$ into $P$ we get that 
$T\PP^3(-1)$ is an extension
of $\Ff _{2,3,P}$ by $\Oo _{\PP ^3}$. 
\end{proof}

Now assume that $c_2(\Ff)=\deg (C)=3$. Since $C$ is reduced and there is no line $D\nsubseteq C$ with $\deg (D\cap C)\ge 3$, we can check that $C$ is connected by checking possible configurations of $C$. In this case, we have $p_a(C)=0$ and $C$ is linearly normal. By the Castelnuovo-Mumford criterion, $\Ii_C(2)$ is spanned.  

In all cases $C$ is arithmetically Cohen-Macaulay and hence $h^1(\Ff (t))=0$ for all $t\in \mathbb {Z}$.
$C$ is not a locally complete intersection if and only if $C$ is the union of three non-coplanar lines through a common point; in this case
$C$ is not Gorenstein; even in this case we have $\deg (\omega _C) =2p_a(C)-2=-2$. Thus in all cases we have $c_3(\Ff )=4$. We also have $h^1(\omega _C(2)) = h^0(C,\Oo _C(-2)) =0$ by the Serre duality for locally Cohen-Macaulay curves and so we have
$h^0(C,\omega _C(2)) = 5$ by Riemann-Roch. Hence $\Ff$ has $\Oo _{\PP ^3}^{\oplus (r-6)}$ as a direct factor if $r\ge 7$. So let us assume that $2\le r\le 6$. 

\begin{lemma}
There is no non-locally free and spanned sheaf $\Ff$ of rank $6$ with no trivial factor whose corresponding curve $C$ has $\deg (C)=3$.
\end{lemma}
\begin{proof}
 We need to analyze the cases when $C$ is reducible. As the worst case let us assume that $C$ is not Gorenstein, i.e. $C=L_1\cup L_2 \cup L_3$, the union of 3 lines $L_i$, not coplanar and through a common point, say $O$. Then $O$ is the only non-Gorenstein point of $C$. We have
$h^0(C,\mathcal{H}om(\omega _C,\omega _C)) = h^1(C,\omega _C) $ by III.7.6 with $n=1$, $i=0$ and $\Ff = \omega _C$ in \cite{Hartshorne}. Since $C$ is connected, we have
$h^0(C,\Oo _C)=1$. Hence the case of $n=i=1$ and $\Ff = \Oo_C$ in III.7.6. \cite{Hartshorne} gives $h^1(C,\omega _C) =1$. We have $0 = h^0(C,\Oo _C(-2))
= h^1(C,\omega _C(2)) $ from the case of $n=1$, $\Ff = \omega _C(2)$, $i=0$ in the duality theorem \cite{Hartshorne}, III.7.6. Hence the Riemann-Roch theorem gives $h^0(C,\omega _C(2)) = \deg (\omega _C(2)) +1=5$.

 Let us fix a point $P\in C$ and let $\Ii _P\omega _C(2)$ denote the image of the natural map $\Ii _P\otimes \omega _C\to \omega _C$ induced by
the inclusion $\Ii _P\hookrightarrow \Oo _C$. To check that $\omega _C(2)$ is spanned at $P$, it is sufficient to prove
$h^1(C,\Ii _P\omega _C(2)) =0$. By \cite{Hartshorne}, III.7.6, we have $h^1(C,\Ii _P\omega _C(2)) = \dim \mbox{Ext}^0(\Ii _P\omega _C(2),\omega _C)$.
Since each element of $H^0(C,\mathcal{H}om(\omega _C,\omega _C))$ is the multiplication by a scalar, while $\Ii _P(2)$ has a non-zero section with at least a zero outside $P$, we get
that the dimension of $ \mbox{Ext}^0(\Ii _P\omega _C(2),\omega _C)$ is zero and thus $\omega _C(2)$ is spanned. 
\end{proof}

\begin{lemma}\label{a21}
There is an indecomposable extension of $\Ff _{2,3,P}$ by $T\PP^3(-1)$.
\end{lemma}
\begin{proof}
From Lemma \ref{lemma}, we can fix a non-split exact sequence 
\begin{equation}\label{eqc+0}
0 \to T\PP^3(-1) \stackrel{\sigma}{\to} \Ff \to \Ff _{2,3,P}\to 0.
\end{equation}
We claim that $\Ff$ is indecomposable. Assume that $\Ff$ is decomposable. $\Ff$ is spanned, because $h^1(T\PP^3(-1))=0$. Since neither
$T\PP^3(-1)$ nor  $\Ff _{2,3,P}$ have a trivial factor, $\Ff$ has no trivial factor. By the classification $c_1=1$ we get
$\Ff \cong T\PP^3(-1)\oplus \Ff _{2,3,P}$. We identify $\Ff$ with the latter sheaf. Write $\sigma = (\sigma _1,\sigma _2)$ with $\sigma _1:T\PP^3(-1) \to T\PP^3(-1)$
and $\sigma_2: T\PP^3(-1) \to \Ff _{2,3,P}$. Since $\sigma$ is injective and $\mbox{rank}(\Ff _{2,3,P}) < \mbox{rank}(T\PP^3(-1))$, we have $\sigma _1\ne 0$. Since
$T\PP^3(-1)$ is simple, $\sigma _1$ is a non-zero multiple of the identity. In particular $\sigma _1$ is invertible. Composing $\sigma$ with the automorphism
\[
\begin{pmatrix}
\sigma _1^{-1}  & 0 \\
-\sigma _2\circ \sigma _1^{-1} & 1
\end{pmatrix}
\]
of $T\PP^3(-1)\oplus \Ff _{2,3,P}$, we get that (\ref{eqc+0}) splits and so we get a contradiction.
\end{proof}

\begin{lemma}\label{n0+}
Fix $P, P'\in \PP^3$ such that $P\ne P'$. Let $V$ be a general 3-dimension linear subspace of $H^0(\Ff _{2,3,P}\oplus \Ff _{2,3,P'})$ and $C\subset \mathbb {P}^3$
be the curve associated to it. Then $C$ is a rational normal curve containing $\{P,P'\}$.
\end{lemma}

\begin{proof}
Since $\Ff _{2,3,P}\oplus \Ff _{2,3,P'}$ is not locally free at $P$ and $P'$, so we have $\{P,P'\}\subset C$. Since $\Ff _{2,3,P}\oplus \Ff _{2,3,P'}$ is spanned
and $\mbox{Sing}(\Ff _{2,3,P}\oplus \Ff _{2,3,P'}) = \{P,P'\}$, we have $\mbox{Sing}(C)\subseteq \{P,P'\}$. 

To prove the lemma it is sufficient to prove
that $C$ is smooth. Assume that $C$ is singular. Let us define $\Omega \subset Gr(3, H^0(\Ff _{2,3,P}\oplus \Ff _{2,3,P'}))$
to be the open subset of the Grassmannian
parametrizing all $V$ for which $C$ is smooth outside $\{P,P'\}$. We also define subsets
$$\Omega (P) := \{W\in \Omega: C \text{ is locally free at }P\}$$ 
and similarly $\Omega(P')$. The subsets $\Omega (P)$ and $\Omega (P')$ are open in $\Omega$. We assumed
$\Omega (P)\cap \Omega (P') =\emptyset$. Since $\Omega$ is irreducible, we have either $\Omega (P) =\emptyset$ or $\Omega (P')=\emptyset$.
Assume $\Omega (P) = \emptyset$ for instance. Since $\Ff _{2,3,P}\oplus \Ff _{2,3,P'} \cong \Ff _{2,3,P'}\oplus \Ff _{2,3,P}$, we get
$\Omega (P')=\emptyset$ and so $\mbox{Sing}(C) =\{P,P'\}$. Let $T$ be the line spanned by $P$ and $P'$. No connected union $C$ with $p_a(C)=0$ of a line and a smooth conic or of $3$ lines
through a common point has exactly two singular points. Let $C$ be a connected union of $3$ lines with $p_a(C) =0$ and with  $\mbox{Sing}(C) =\{P,P'\}$. We have
$C = T\cup L\cup L'$ with $L$ a line through $P$ and $L'$ a line through $P'$. Let us fix a point $O\in T\setminus \{P,P'\}$. Since $\Ff _{2,3,P}\oplus \Ff _{2,3,P'}$ is spanned
and $O\notin V(\Ff _{2,3,P}\oplus \Ff _{2,3,P'})$ for a general $V\in \Omega$, so the associated curve, say $C_1$, does not contain $O$. Hence $T\nsubseteq C_1$ and we get a contradiction.
\end{proof}

\begin{lemma}\label{n1+}
Fix $P, P'\in \PP ^3$ and a rational normal curve $C\subset \PP^3$. Then $\Ff _{2,3,P}\oplus \Ff _{2,3,P'}$
is the unique spanned reflexive sheaf $\Ff$ of rank $4$ with $c_1=2$, $c_3 =4$, $\mathrm{Sing}(\Ff )
= \{P,P'\}$ and associated to $C$.
\end{lemma}

\begin{proof}
Let $\Delta$ be the set of all triples $(C_1,P_1,P'_1)$ with $C_1$ a rational normal curve, $P_1\in C_1$, $P'_1\in C_1$ and $P_1\ne P'_1$. Any two triples in $\Delta$ are projectively equivalent. Hence the lemma is true for one triple $(C_1,P_1,P'_1)\in \Delta$ if and only if it is true for all such triples.
Since any two elements of $\Delta$ are projectively equivalent,  there is an exact sequence (\ref{eqa1}) with $r=4$, $ \Ff _{2,3,P}\oplus \Ff _{2,3,P'}$ instead of $\Ff$, $C=C_1$ a rational normal curve
and with $\{P,P'\} = \mbox{Sing}(\Ff _{2,3,P}\oplus \Ff _{2,3,P'})$. Hence $(C,P,P')$ is realized by some extension (\ref{eqa1}) with
$r=4$ and $\Ff = \Ff _{2,3,P}\oplus \Ff _{2,3,P'}$.

 Let us fix a triple $(C, P,P')$. A spanned reflexive sheaf
$\Ff$ comes from (\ref{eqa1}) with $r=4$ and $\mbox{Sing}(\Ff )
= \{P,P'\}$ if and only if it is associated to a 3-dimensional linear subspace of $H^0(C,\omega _C(2))$ spanning $\omega _C(2)$ at
all points of $C\setminus \{P,P'\}$, but spanning $\omega _C(2)$ neither at $P$ nor at $P'$. Since $\omega _C(2)$ has degree $4$, we have
$h^0(C,\mathcal {I}_{\{P,P'\}}\otimes \omega _C(2)) =3$. Hence there is a unique reflexive sheaf $\Ff$ obtained in this way and so $\Ff \cong \Ff _{2,3,P}\oplus \Ff _{2,3,P'}$
by Lemma \ref{n0+}.
\end{proof}

\begin{proposition}\label{n2+}
$\Ff _{2,3,P}\oplus \Ff _{2,3,P'}$ is the unique spanned reflexive sheaf on $\PP^3$ with no trivial factor, $(c_1, c_2, c_3)=(2,3,4)$ and $\{P,P'\}$ as its non-locally free locus.
\end{proposition}

\begin{proof}
Let $\Ff$ be a spanned reflexive sheaf of rank $4$ with $(c_1, c_2, c_3)=(2,3,4)$, $\mbox{Sing}(\Ff )=\{P,P'\}$ and no trivial factor. Assume
$\Ff \ne \Ff _{2,3,P}\oplus \Ff _{2,3,P'}$ and then $\Ff$ is indecomposable by Proposition \ref{oo1}. Take a general 3-dimensional
linear subspace of $H^0(\Ff)$ and use it to get the exact sequence (\ref{eqa1}) with $r=4$. The curve $C$ is smooth outside $\{P,P'\}$ and we know that $\Ff$
comes from a 3-dimensional linear subspace of $H^0(\omega _C(2))$ spanning $\omega _C(2)$ at all points of $C\setminus \{P,P'\}$, but
spanning $C$ neither at $P$ nor at $P'$. If $C$ is a rational normal curve, then we can apply Lemma \ref{n1+} and thus we may assume that $C$ is reducible. The Bertini
theorem gives $\mbox{Sing}(C) \subseteq \{P,P'\}$. Exchanging $P$ and $P'$ if necessary, we may assume that $C$ is singular at $P$.

\quad (a) First assume that $C$ is not formed by 3 lines through $P$. Hence $C$ is Gorenstein and $\omega _C(2)$ is a degree $4$ spanned line
bundle. First assume $h^1(C,\mathcal {I}_{\{P,P'\}}\otimes \omega _C(2)) =0$ and so we have $h^0(C,\mathcal {I}_{\{P,P'\}}\otimes \omega _C(2)) =3$. Then the triple $(C,P,P')$
gives a unique bundle $\Gg$. There is an integral affine curve $A$ with $o\in A$ and a flat family of
triples $\{(C_t, P_t, P'_t)\}_{t\in A}$ with $C_t\subset \mathbb {P}^3$ a smooth rational normal curve containing $\{P,P'\}$ for
all $t\in A\setminus \{o\}$ and $P_t$, $P'_t$ distinct points of $C_t$. In this family the
function $t\mapsto h^0(C_t,\mathcal {I}_{\{P,P'\}}\otimes \omega _{C_t}(2))$ is constant. 

Hence we find a flat family of spanned reflexive sheaves $\{\Ff _t\}_{t\in A}$ with $\Ff_o
\cong \Ff$ and $\mbox{Sing}(\Ff _t) =\{P,P'\}$ for all $t$. Lemma \ref{n1+} gives $\Ff _t \cong \Ff _{2,3,P}\oplus \Ff _{2,3,P'}$ for all $t\ne o$.
By the semicontinuity theorem for cohomology we get the existence
of non-zero maps 
\begin{align*}
&m: \Ff _{2,3,P} \to \Ff~~~,~~~ m': \Ff _{2,3,P'} \to \Ff\\
&u: \Ff \to \Ff _{2,3,P}  ~~~,~~  u': \Ff \to \Ff _{2,3,P'} 
\end{align*}
 with $u'\circ m=0$ and $u\circ m' =0$.
Since $\Ff _{2,3,P}$ and $\Ff _{2,3,P'}$ are stable and non-isomorphic, we get $\Ff \cong \Ff _{2,3,P}\oplus \Ff _{2,3,P'}$.

\quad (b) Now assume that $C$ is formed by 3 non-coplanar lines through a common point $P$, which is the unique singular point
of $C$. By step (a) and Lemma \ref{n1+} to find a contradiction to the irreducibility of $\Ff$ it is sufficient to find another exact sequence (\ref{eqa1}) for $\Ff$ with $C$ not
the union of 3 lines through either $P$ or $P'$. Let $T\subset \mathbb {P}^3$ be the line spanned by $P$ and $P'$. It is contained in any union $C_1$ of 3 lines through
a common point containing both $P$ and $P'$. Let us fix $O\in T\setminus \{P,P'\}$. Since $\Ff$ is spanned
and $O\notin V(\Ff )$ for a general 3-dimensional subspace of $H^0(\Ff )$, the associated curve, say $C_1$, does not contain $O$. Thus we have $T\nsubseteq C_1$.
Since $\{P,P'\} = \mbox{Sing}(\Ff ) \subset C_1$, we get that $C_1$ is not a union of 3 lines through a common point.
\end{proof}

\begin{proposition}\label{p2}
For each $r\in \{2,\dots, 5\}$ there exists a non-locally free, spanned and reflexive sheaf $\Ff$ of rank $r$ and with no trivial factor, $(c_1, c_2)=(2,3)$ and $k:=\sharp(\mathrm{Sing}(\Ff))= 6-r$. 
Moreover, there exists an indecomposable sheaf in each cases except when $(r,k)=(4,2)$. 
\end{proposition}

\begin{proof}
Let us take $C$ a normal rational curve of degree 3. For any $S\subset C$ with $\sharp (S)=k$, we have $h^0(C, \omega_C(2)\otimes \Oo_C(-S))=5-k$ and the linear system $|\omega_C(2)\otimes \Oo_C(-S)|$ has $S$ as its base locus. Thus $V:=H^0(\omega_C(2)\otimes \Oo_C(-S))\subset H^0(\omega_C(2))$ defines an extension 
$$0\to V\otimes \Oo_{\PP^3} \to \Ff \to \Ii_C(2) \to 0$$
with $\mbox{Sing}(\Ff)=S$. Such $\Ff$ is reflexive and spanned since $\Ii_C(2)$ is spanned. In other words, we obtain the first statement due to Remark \ref{b1}. 

For the case of $k=1$, we have an indecomposable sheaf of rank $5$ with one singular point by Lemma \ref{a21}. When $k=2$, i.e. $r=4$, we have a unique choice for such sheaves that is $\Ff_{2,3,P}\oplus \Ff_{2,3,P'}$ by Proposition \ref{n2+}. Thus we cannot have an indecomposable one in this case. 
When $k\in \{3,4,5\}$, every such sheaves are indecomposable due to Proposition \ref{oo1} and Remark \ref{lem}.
\end{proof}

In fact, we can have indecomposable reflexive sheaves $\Ff$ of rank $r$ even with $\sharp (\mathrm{Sing}(\Ff))<6-r$. For example, when $r=4$, we have the following result :
\begin{lemma}
There exists an indecomposable, reflexive and spanned sheaf $\Ff$ of rank $4$ with $\sharp (\mathrm{Sing}(\Ff))=1$. 
\end{lemma}
\begin{proof}
Let us fix a general map $\sigma : \Oo_{\PP^3} \to \Ff_{2,3,P}\oplus T\PP^3(-1)$. Then the sheaf $\Ff:=\mathrm{Coker}(\sigma)$ is locally free outside $P$, because $\Ff _{2,3,P}\oplus T\PP^3(-1)$ is spanned and it is locally free outside $P$.
Since the non-locally free sheaf  $\Ff _{2,3,P}\oplus T\PP^3(-1)$ is an extension of $\Ff$ by $\Oo _{\PP^3}$, $\Ff$ is not locally free at $P$. We also get that $\Ff$ is  a spanned reflexive sheaf whose corresponding curve $C$ has degree $3$. 

Assume that $\Ff$ is decomposable. Since we have $\mbox{Sing}(\Ff ) =\{P\}$, the classification in Proposition \ref{oo1} gives that either
$\Ff \cong \Ff _{2,3,P}\oplus \Ff _{2,3,P}$ or $\Ff \cong \Ff _{2,3,P}\oplus \Oo_{\PP^3}(1)\oplus \Oo_{\PP^3}$. The latter case is excluded, because $\Ff _{2,3,P}$ and $\Oo_{\PP^3}(1)\oplus \Oo_{\PP^3}$
have different Segre classes and hence $\Ff $ and $\Ff _{2,3,P}\oplus \Oo_{\PP^3}(1)\oplus \Oo_{\PP^3}$ have different Segre classes. About the first case, let us note that the fibers of $\Ff _{2,3,P}\oplus \Ff _{2,3,P}$
and of $\Ff _{2,3,P}\oplus T\PP^3(-1)$  at $P$ have
dimension $6$. We fixed the point $P$ and then took a general $\sigma$. Since $\sigma$ is general, the fiber of $\Ff$ at $P$ has dimension $5$. Hence
$\Ff$ is not isomorphic to $\Ff _{2,3,P}\oplus \Ff _{2,3,P}$. Thus we have an indecomposable, reflexive and spanned sheaf of rank 4 with one singular point. 
\end{proof}

Now assume that $c_2(\Ff)=\deg (C)=4$ as the final case. Since $\Ii _C(2)$ is spanned outside finitely many points, we get that $C$ is a complete intersection and so $\Ii_C(2)$ is spanned. Thus $\Ff$ is spanned, $h^0(\Ff )=r+1$
and $\omega _C(2) \cong \Oo _C(2)$. Since $h^0(C, \omega_C(2))=8$, we have $\Ff \cong \Gg \oplus \Oo_{\PP^3}^{\oplus (r-9)}$ for some $\Gg$ if $r\ge 10$ and so we can analyze the cases $2\le r \le 9$. Since the case when $\Ff$ is locally free was already dealt in \cite{SU}, we will assume that $\Ff$ is non-locally free. If $r=9$, then $V=H^0(C, \omega_C(2))$ spans $\omega_C(2)$  and so $\Ff$ is locally free. Thus we can assume that $2\le r\le 8$. 

\begin{lemma}\label{b2}
Let $C$ be an irreducible curve and $\Ll$ a very ample line bundle on $C$. For a fixed integer $k$ with $1\le k \le h^0(\Ll)-1$ and a set $A$ of general $k$ points on $C$, we have $h^0(\Ll (-A))=h^0(\Ll)-k$ and the base locus of the linear system $|\Ll (-A)|$ is finite. Moreover, if $k$ satisfies $1\le k \le h^0(\Ll)-2$, then $A$ is exactly the base locus of $|\Ll (-A)|$ for a general $A$.  
\end{lemma}

\begin{proof}
Since $\Ll$ is very ample, we can embed $C$ as a linearly normal curve $C\subset \mathbb {P}^m$ with $m:= h^0(\Ll)-1$ by the linear system of $\Ll$. Any set of general $k$ points with $k\le m$ on the non-degenerate
curve $C\subset \mathbb {P}^m$ are linearly independent and any hyperplane contains only finitely many points of $C$. Hence $h^0(\Ll(-A)) = h^0(\Ll)-k$ and the base locus of the linear system $\vert \Ll(-A)\vert$
is finite.

 If $k\le m-1$, then $A$ is the base locus
of the linear system $\vert \Ll(-A)\vert$, because a general hyperplane section of $C$ is in linearly general position (see page 109 in \cite{acgh}).
\end{proof}

\begin{proposition}
For each $r\in \{2, \dots, 8\}$ there exists a non-locally free, spanned, indecomposable and reflexive sheaf $\Ff$ with $(c_1,c_2)=(2,4)$ and 
$$\sharp (\mathrm{Sing}(\Ff)) = \left\{
\begin{array}{ll}
8, &\hbox{if $r=2$};\\
9-r, &\hbox{if $r\ne 2$}.
\end{array}
\right.$$

\end{proposition}

\begin{proof}
Note that $\omega_C(2)$ is very ample and $h^0(\omega_C(2))=8$. Assume that $C$ is irreducible. Then for each integer $k$ with $1\le k \le 6$, we can find a finite subset $S\subset C$ such that $\sharp(S)=k$ and the base locus of the linear system $|\omega_C(2)\otimes \Oo_C(-S)|$ is finite and so $V:=H^0(\omega_C(2)\otimes \Oo_C(-S))$ defines a reflexive sheaf $\Ff$ of rank $9-k$ with no trivial factor. We also have $\mbox{Sing} (\Ff)=S$ by the second statement of Lemma \ref{b2}. In the case of $r=2$, the assertion follows by Theorem 4.1 in \cite{Hartshorne1}. 

Now as in the proof of Proposition \ref{p2}, all such sheaves are indecomposable due to the classification of sheaves with $c_1=1$. 
\end{proof}

So far we give a complete description on reflexive sheaves on $\PP^3$ with $c_1(\Ff)=2$ and $\dim V(\Ff)\le 0$. As an automatic consequence, we obtain that every such a reflexive sheaf on $\PP^3$ is really spanned everywhere, i.e. $V(\Ff)=\emptyset$ (see Corollary \ref{OOO2}).

\begin{remark}
In \cite{SU} we have a complete list of the spanned bundles on $\PP^n$ with $c_1\le 2$ and so the possibilities for decomposable spanned bundles with no trivial factor and $c_1=2$ are $\Oo_{\PP^n}(1)\oplus T\PP^n(-1)$ and $T\PP^n(-1)^{\oplus 2}$. Conversely, let $\Ee$ be a spanned bundle of rank 4 on $\PP^3$ with $(c_1, c_2, c_3)=(2,2,2)$ and no trivial factor. Then it fits into the following sequence
$$0\to \Oo_{\PP^3}^{\oplus 3} \to \Ee \to \Ii_C(2) \to 0$$
where $C$ is a smooth conic in $\PP^3$. Applying Theorem 2.3 in \cite{AO} to $\Ee^\vee$ with $j=1$ and $t=0$, we obtain that $\Omega_{\PP^3}(1)$ is a direct summand of $\Ee^\vee$ and so $\Ee\cong \Oo_{\PP^3}(1)\oplus T\PP^3(-1)$. Similarly every spanned bundle of rank 6 on $\PP^3$ with $(c_1, c_2, c_3)=(2,3,4)$ and no trivial factor is isomorphic to $T\PP^3(-1)^{\oplus 2}$. It gives us the list of indecomposable bundles from Theorem 1.1 of \cite{SU}. 
\end{remark}

%%%%%%%%%%%%%%%%%%%%%%%%%%%%%%%%%%%

\section{Locally free sheaf with $c_1=2$ on $\PP^n$, $n\ge 3$}

In this section, we prove a similar statement to Corollary \ref{OOO2} for vector bundles on $\PP^n$. 

\begin{theorem}\label{xx1}
Let $\Ee$ be a vector bundle on $\PP ^n$ with $n\ge 3$ such that $c_1(\Ee )=2$ and assume that $\Ee$ is spanned outside a subvariety of $\PP^n$ with codimension at least $3$. Then
$\Ee$ is spanned and so it is one of the vector bundles described in \cite{SU}.
\end{theorem}

\begin{proof}
 The case $n=3$ is true by Corollary \ref{OOO2}. 
 
 Now assume $n>3$ and $V(\Ee) \ne \emptyset$. For a fixed point $P\in V(\Ee )$, let $M\subset \PP^n$
be a general 3-dimensional linear subspace containing $P$. Since $M$ is general, $M\cap V(\Ee)$ has dimension $\max \{0,\dim (V(\Ee))-3\}$. Set $\Ff := \Ee \vert _M$. By Corollary \ref{OOO2} the vector bundle $\Ff$ is spanned
and hence it is as listed in \cite{SU} for the case $n=3$.  Let $V\subseteq H^0(\Ff )$ denote the image of the restriction
map $H^0(\Ee )\to H^0(\Ff)$. Taking $r-1$ general sections of $H^0(\Ee)$, we get an exact sequence
\begin{equation}\label{eqb1}
0 \to \Oo_{\PP ^n}^{\oplus (r-1)} \to \Ee \to \Ii _K(2)\to 0
\end{equation}
where $K$ is a locally Cohen-Macaulay subscheme of $\PP^n$ of pure codimension 2 (or the empty set). 

By part (F) at page 4 in \cite{am} and their references \cite{B}\cite{Chang}, $K$ is smooth outside $V(\Ee)$ and a subscheme $\Delta$ of codimension at least $4$. Let $C\subset M$ denote the scheme-theoretic intersection $K\cap M$. Since the tensor product is a right exact functor, the ideal sheaf $\Ii _{C,M}(2)$
is the image of
$\Ii _K(2)\otimes \Oo _M$ inside $\Oo _M$. We often write $\Ii _C$ instead of $\Ii _{C,M}$ when we are working with $\Ff$ and $M$. Since $\Oo _{\PP ^n}^{\oplus (r-1)}$ has no torsion, by tensoring (\ref{eqb1}) with $\Oo _M$ we get the exact sequence on $M$:
\begin{equation}\label{eqb2}
0 \to \Oo _M^{\oplus (r-1)} \to \Ff \to \Ii _C(2)\to 0
\end{equation}
(here we use that $\Ee$ is locally free, otherwise $\Ee \otimes \Oo _M$ may have torsion).  Since $M$ is a general $3$-dimensional
subspace containing $P$, we have $M\cap (V(\Ee )\cup \Delta )=\{P\}$. Hence the Bertini theorem gives that $C$ is smooth outside $P$. By (\ref{eqb2}) the germ $\Ii _{C,P}(2)$
has a free resolution of length $2$. Since $\Oo _{M,P}$ is a regular local ring of dimension $3$, so $\Oo _{C,P}(2)$ has depth $1$, i.e. $C$ is locally Cohen-Macaulay at $P$.
Since $C$ is reduced outside $P$, we get that $C$ is reduced at $P$. The integer $\deg (C)$ is the degree of the union (counting multiplicities) of the codimension $2$ components of $K$.

\quad (a) If $\Ff \cong \Oo _{\PP ^3}(2)\oplus \Oo _{\PP ^3}^{\oplus (r-1)}$, then by Corollary 1.7 in \cite{ef} we have $\Ee \cong \Oo_{\PP^n}(2)\oplus \Oo_{\PP^n}^{\oplus (r-1)}$. In particular, it is spanned. Similarly when $\Ff\cong \Oo _{\PP ^3}(1)^{\oplus 2}\oplus \Oo _{\PP ^3}^{\oplus (r-2)}$, then we have $\Ee  \cong \Oo _{\PP ^n}(1)^{\oplus 2}\oplus \Oo _{\PP ^n}^{\oplus (r-2)}$.

\quad (b) Now assume $\Ff \cong \Nn_{\PP^3}(1)\oplus \Oo _{\PP^3}^{\oplus (r-2)}$, where $\Nn_{\PP^3}$ is a null-correlation bundle. Hence $C$ is a degeneration of a family of pairs
of disjoint lines. Since $C$ is reduced, it is a disjoint union of two lines. Hence $K$ is (outside $V(\Ee )$) the union of two $2$-codimensional linear subspaces, $K_1$ and $K_2$
with $K_1\cap K_2\cap M = \emptyset$. Since $M$ is general among the 3-dimensional linear subspaces containing $P$, we get $P\notin K_1\cap K_2$. $K$ was fixed. Then
we take any $P\in V(\Ee)$. We get $K_1\cap K_2 \cap V(\Ee )=\emptyset$ and that $K$ is smooth at each point of $V(\Ee)$. Since $\Ee$ is locally free
and for each $P'\in \PP^n$ the ring $\Oo _{\PP ^n,P'}$ is an $n$-dimensional local ring, so (\ref{eqb1}) implies
that $K$ has depth $\ge n-2$. Fix $P'\in \mbox{Sing}(K)$. Taking a general 3-dimensional linear subspace $M'$ containing $P'$ we get as above that $K\cap M'$ is a disjoint union of two lines. Hence $K\cap M'$ is smooth at $P'$. Hence $K$ is smooth at $P'$, a contradiction. Since $n\ge 4$, we have $K_1\cap K_2\ne \emptyset$. Hence $\mbox{Sing}(K)\ne \emptyset$. Hence this case does not occur if $n\ge 4$. Alternatively, since $K$ is reduced with pure codimension $2$ and $K\cap M$ is the disjoint union
of two lines, $K$ is a union of two linear subspaces $K_1$ and $K_2$ such that $\dim (K_1)=\dim (K_2) =n-2$ and $\dim (K_1\cap K_2)=n-4$. $K$ is not locally Cohen-Macaulay
at any point of $K_1\cap K_2$.

\quad (c) Now assume $\Ff \cong \Omega _{\PP^3}(2)\oplus \Oo _{\PP^3}^{\oplus (r-3)}$. As in (b) we get a contradiction for all $n\ge 4$, because even in this case $C$ is the disjoint union of two lines. Indeed, we have $h^1(\Ii_C)=h^1(\Omega_{\PP^3})=1$ and it implies that $h^0(\Oo_C)=2$ and so $C$ is a curve of degree $2$ with two connected components. 

\quad (d) Now assume that $\Ff$ fits in an exact sequence
$$0\to \Oo _{\PP^3}(-2) \to \Oo _{\PP^3}^{\oplus (r+1)} \to \Ff \to 0.$$
Note that $h^0(\Ff )=r+1$ and no proper subspace of $H^0(\Ff)$ spans $\Ff$ outside finitely many
points. Hence the restriction map $H^0(\Ee )\to H^0(\Ff )$ is surjective and so $P\notin V(\Ee )$, a contradiction.

\quad (e) Now assume that $\Ff$ fits in an exact sequence
$$0\to \Oo _{\PP^3}(-1)^{\oplus 2} \to \Oo _{\PP^3}^{\oplus( r+2)} \to \Ff \to 0.$$
Note that $h^0(\Ff )=r+2$ and the image of $V$ in $H^0(\Ii _C(2))$ has dimension $2$. Thus its zero-locus is a degree $4$ complete intersection containing $C$. Since $V_1$ spans $\Ii _C(2)$ outside finitely many points, we get a contradiction.
\end{proof}

\providecommand{\bysame}{\leavevmode\hbox to3em{\hrulefill}\thinspace}
\providecommand{\MR}{\relax\ifhmode\unskip\space\fi MR }
\providecommand{\MRhref}[2]{%
  \href{http://www.ams.org/mathscinet-getitem?mr=#1}{#2}
}
\providecommand{\href}[2]{#2}

\end{document}